\documentclass[11pt,reqno]{amsart}
\usepackage{amsthm,amsmath,amssymb,mathrsfs}
\usepackage{pstricks,pst-node,pst-text,pst-3d,pst-plot}
\newrgbcolor{verde}{0 .7 .1}
\newrgbcolor{verdecito}{.85 1 .7}
\newrgbcolor{rojo}{.8 0 .1}

\theoremstyle{plain}
\newtheorem{theorem}{Theorem} [section]
\newtheorem{corollary}[theorem]{Corollary}
\newtheorem{lemma}[theorem]{Lemma}
\newtheorem{proposition}[theorem]{Proposition}

\theoremstyle{definition}
\newtheorem{definition}[theorem]{Definition}

\newtheorem{remark}[theorem]{Remark}

\newtheorem{example}[theorem]{Example}

\def\N{{\mathbb{N}}}
\def\R{{\mathbb{R}}}
\def\Z{{\mathbb{Z}}}

\def\clspan{{\overline{\mathrm{span}}}}

\def\fhat{{\hat{f}}}

\newcommand{\vp}{\varphi}
\newcommand{\w}{\omega}

\newcommand{\W}{\Omega}

\newcommand{\s}{\sigma}
\newcommand{\SR}{\mathcal{D}}
\newcommand{\SN}{\mathcal{N}}
\newcommand{\sn}{\overline{\text{span}}}

 \theoremstyle{plain}
 
 \theoremstyle{plain}
  
 \theoremstyle{plain}
  
 \theoremstyle{definition}
 
 \theoremstyle{remark}

\newtheorem{thm*}{Teorema}[section]

\newcommand{\CHI}{\hbox{\raise .4ex \hbox{$\chi$}}}

\newcommand{\set}[1]{\{#1\}}
\newcommand{\bigset}[1]{\bigl\{#1\bigr\}}

\newcommand{\dotoplus}{\operatornamewithlimits{\dot{\oplus}}}
\newcommand{\dotbigoplus}{\operatornamewithlimits{\dot{\bigoplus}}}

\def\subset{\subseteq}
\def\supp{{\textrm{supp}}}

\def\clspan{{\overline{\mathrm{span}}}}

\def\fhat{{\widehat{f}}}

\def\EQ{\, = \,}

\begin{document}

\title{Invariance of a Shift-Invariant Space in Several Variables}

\author[M. Anastasio, C. Cabrelli, and V. Paternostro]{M. Anastasio,
C. Cabrelli, and V. Paternostro}

\address{\textrm{(M. Anastasio)}
Departamento de Matem\'atica,
Facultad de Ciencias Exac\-tas y Naturales,
Universidad de Buenos Aires, Ciudad Universitaria, Pabell\'on I,
1428 Buenos Aires, Argentina and
CONICET, Consejo Nacional de Investigaciones
Cient\'ificas y T\'ecnicas, Argentina}
\email{manastas@dm.uba.ar}

\address{\textrm{(C. Cabrelli)}
Departamento de Matem\'atica,
Facultad de Ciencias Exac\-tas y Naturales,
Universidad de Buenos Aires, Ciudad Universitaria, Pabell\'on I,
1428 Buenos Aires, Argentina and
CONICET, Consejo Nacional de Investigaciones
Cient\'ificas y T\'ecnicas, Argentina}
\email{cabrelli@dm.uba.ar}

\address{\textrm{(V.Paternostro)}
Departamento de Matem\'atica,
Facultad de Ciencias Exac\-tas y Naturales,
Universidad de Buenos Aires, Ciudad Universitaria, Pabell\'on I,
1428 Buenos Aires, Argentina and
CONICET, Consejo Nacional de Investigaciones
Cient\'ificas y T\'ecnicas, Argentina}
\email{vpater@dm.uba.ar}

 \thanks{The research of
 the authors is partially supported by
Grants: ANPCyT, PICT 2006--177, CONICET, PIP 5650, UBACyT X058 and X108.
}



\date{\today}
\maketitle
\begin{abstract}
In this article we study invariance properties of shift-invariant spaces in higher dimensions.
We state and prove several necessary and sufficient conditions for a shift-invariant space to be invariant under a given closed subgroup of $\R^d$, and prove the existence of shift-invariant
spaces that are exactly invariant  for  each given subgroup. As an application we  relate the extra invariance to the size of support of the Fourier transform of the generators of the shift-invariant space.

This work extends recent results obtained for the case of one variable to several variables.

\end{abstract}

\section{Introduction}

A {\it shift-invariant space} (SIS) of $L^2(\R)$ is a closed subspace that is invariant
under translations by integers.
These spaces are important in approximation theory, wavelets, sampling and frames.
They also serve as  models in many applications in signal and image processing.

An important and interesting question regarding these spaces is whether they have the property to be invariant under translations other than integers.
A limit case is when the space is invariant under translations by all real numbers. In this case the space is called {\it translation invariant}.
However there exist  shift-invariant spaces with some {\it extra} invariance that are not necessarily translation invariant. That is, there are some intermediate cases between shift-invariance and translation invariance. The question is then, how can we identify them?

Recently, Hogan and Lakey defined the {\it discrepancy} of a shift-invariant space as a way to quantify
the {\it non-translation invariance} of the subspace, (see  \cite{HL05}). The discrepancy measures how far a unitary norm function of the subspace, can move away from it, when translated by non integers.
A translation invariant space has discrepancy zero.

In another direction, Aldroubi et al, see \cite{ACHKM} studied  shift-invariant spaces of $L^2(\R)$
that have some extra invariance. They  show that if $S$ is a shift-invariant space, then its {\it invariance set}, is a closed additive subgroup of $\R$  containing $\Z.$ The invariance set associated to  a shift-invariant space is the set $M$ of real numbers  satisfying that  for each $p \in M$ the translations by $p$ of every function in $S$,  belongs to $S$.
As a consequence, since every additive subgroup of $\R$ is either discrete or dense,  there are only two
possibilities left for the extra invariance.  That is, either  $S$ is invariant under translations by the group $(1/n)\Z$, for some positive integer $n$ (and not invariant under any bigger subgroup) or it is translation invariant.
They found different characterizations, in terms of the Fourier transform, of when a shift invariant space is $(1/n)\Z$-invariant.

A natural question  arises in this context.  Are the characterizations of extra invariance that hold on the line,  still valid in several variables?

A shift-invariant space in $L^2(\R^d)$ is a closed subspace that is invariant under translations
by the group $\Z^d.$
 The invariance set  $M \subset \R^d$  associated to a shift-invariant space $S$, that is, the set of vectors that leave $S$ invariant when translated  by its elements, is again, as in the 1-dimensional case,  a closed subgroup of $\R^d$ (see Proposition \ref{M-isgroup}).
 The problem of the extra invariance can then be  reformulated  as finding necessary and sufficient
 conditions for   a shift-invariant space to be  invariant
under a closed additive subgroup $M \subset \R^d$ containing $\Z^d.$

The main difference here with the one dimensional case, is that the structure of the subgroups of $\R^d$ when  $d$ is bigger than one, is not as simple.

The results obtained for the 1-dimensional case translate very well in the case in which the invariance set  $M$  is a lattice, (i.e. a discrete group)
or when $M$ is dense, that is  $M=\R^d$.
However, there are  subgroups of $\R^d$ that are neither discrete nor dense.
So, can there exist  shift-invariant spaces which are $M$-invariant for  such a subgroup $M$ and are not translation invariant?

 In this paper we studied the extra invariance of shift-invariant spaces in higher dimensions.
We were able to  obtain several characterizations paralleling  the 1-dimensional results. In addition our results show the existence of shift-invariant spaces that are {\it exactly} $M$-invariant for every closed subgroup $M\subset  \R^d$ containing $\Z^d$.
 By `exactly $M$-invariant' we mean that they are not invariant under any other subgroup containing $M.$
 We apply our results to obtain estimates on the size of the support of the Fourier transform of the generators of the space.

 The paper is organized in the following way.
Section \ref{sec-1} contains some notations, definitions and preliminary results that will be needed throughout. We review the structure of closed additive subgroups of $\R^d$ in
Section \ref{sec-2}. In Section \ref{sec-3} we extend some results, known for shift-invariant spaces in $\R^d$, to $M$-invariant spaces when $M$ is a closed subgroup of $\R^d$ containing $\Z^d.$
 The necessary and sufficient conditions for the $M$-invariance of shift-invariant spaces are
 stated and proved in Section \ref{sec-4}. Finally, Section \ref{sec-5} contains some applications of our results.

\section{Preliminaries}\label{sec-1}

\subsection{Notation and Definitions}\noindent

Given a subspace $W$ of a Hilbert space $H$, we
denote by $\overline{W}$ its closure and by $W^{\perp}$
its orthogonal complement.
The inner product in $H$ will be denoted by $\langle\cdot,\cdot\rangle$.

We normalize the Fourier transform of $f \in L^1(\R^d)$ as
$$\fhat(\omega) \EQ \int_{\R^d} f(x) \, e^{-2\pi i\langle\omega, x\rangle} dx.$$
The Fourier transform extends to a unitary operator on $L^2(\R^d)$.
Given $\Phi \subseteq L^2(\R^d)$, we set
$\widehat \Phi = \set{\fhat : f \in \Phi}$.

For $y\in\R^d$, we write $e^{-2\pi i \langle \omega, y\rangle} $ as $e_y(\w)$ and the translation operator $t_y$ is $t_y f(x) = f(x-y)$.
Note that $\widehat{(t_y f)}(\omega) = e_y(\w)\fhat(\omega)$.

Let $G$ be a subset of $\R^d$, we will say that a function $f$ defined in $\R^d$ is $G$-periodic
if $t_xf=f$ for all $x\in G$. A subset $A\subset\R^d$ is $G$-periodic if its indicator function (denoted
by $\chi_A$) is $G$-periodic.

A \emph{shift-invariant space} (SIS) is a closed subspace $S$ of $L^2(\R^d)$
such that $t_kf\in S$ for every $k\in\Z^d$ and $f\in S$.

Given $\Phi \subseteq L^2(\R^d)$, we define
$$E(\Phi)
\EQ \set{t_k \varphi: \varphi \in \Phi, \, k \in \Z^d}.$$
The SIS generated by $\Phi$ is
$$S(\Phi)
\EQ \clspan\,E(\Phi)
\EQ \clspan\set{t_k \varphi: \varphi \in \Phi, \, k \in \Z^d}.$$
We call $\Phi$ a \emph{set of generators} for $S(\Phi)$.
When $\Phi = \set{\varphi}$, we simply write $S(\varphi)$.

The \emph{length} of a shift-invariant space $S$ is the minimum cardinality of the
sets $\Phi$ such that $S = S(\Phi)$.
A SIS of length one is called a \emph{principal} SIS.
A SIS of finite length is a \emph{finitely generated} SIS (FSIS).

We will write $W = U \dotoplus V$ to denote the \emph{orthogonal} direct
sum of closed subspaces of $L^2(\R^d)$,
i.e., the subspaces $U$, $V$ must be closed and orthogonal, and~$W$
is their direct sum.

The Lebesgue measure of a set $E \subseteq \R^d$ is denoted by $|E|$.

The cardinality of a finite set $F$ is denoted by $\#F$.

\subsection{The invariance set}\noindent

Let $S\subset L^2(\R^d)$ be a SIS. We define the set
\begin{equation}\label{M}
M:=\{x\in\R^d\,:\,t_x f\in S,\,\,\,\forall\,f\in S\}.
\end{equation}
If $\Phi$ is a set of generators for $S$, it is easy to check that,  $x\in M$ if and only if  for all $\varphi\in \Phi$, $t_x\varphi\in S$.

In case that $M=\R^d$, Wiener's theorem (see \cite{Hel64}, \cite{Sri64}) states that there exists
a measurable set $E\subset\R^d$
satisfying $$S=\{f\in L^2(\R^d)\,:\, \supp(\widehat{f}\,)\subseteq E\}.$$

We want to characterize $S$ when $M$ is not all $\R^d$. We will first study the structure of the set $M$.

\begin{proposition}\label{M-isgroup}

Let $S$ be a SIS of $L^2(\R^d)$ and let $M$ be defined as in (\ref{M}).
Then $M$ is an additive closed subgroup of $\R^d$ containing $\Z^d$.

\end{proposition}

For the proof of this proposition we will need the following lemma.
Recall that an additive  semigroup is a non-empty  set with an associative additive operation.
\begin{lemma}\label{semi}
Let $H$ be a closed semigroup of $\R^d$ containing $\Z^d$, then $H$ is a group.
\end{lemma}

\begin{proof}

Let $\pi$ be the quotient map from $\R^d$ onto $\mathbb{T}^d=\R^d/\Z^d$.
Since  $H$ is a semigroup containing $\Z^d$, we have that $H+\Z^d=H$, thus
$$
\pi^{-1}(\pi(H))=\bigcup_{h\in H}h+\Z^d=H+\Z^d=H.
$$
This shows that $\pi(H)$ is closed in $\mathbb{T}^d$ and therefore compact.

By \cite[Theorem 9.16]{HR79}, we have that a compact semigroup of $\mathbb{T}^d$ is necessarily a group,
thus  $\pi(H)$ is a group and consequently $H$ is a group.

\end{proof}

\begin{proof}[\textit{ Proof of Proposition \ref{M-isgroup}}]

Since $S$ is a SIS, $\Z^d\subset M$.

We now show that $M$ is closed.
Let $x_0\in \R^d$ and $\{x_n\}_{n\in \N}\subset M$,
such that  $\lim_{n\to\infty}x_n=x_0$.

Then
$$\lim_{n\to\infty}\|t_{x_n}f-t_{x_0}f\|=0.$$
Therefore, $t_{x_0}f\in \overline{S}$.
But $S$ is closed, so $t_{x_0}f\in S$.

It is easy to check that $M$ is a semigroup of $\R^d$, hence we conclude from Lemma \ref{semi}
that $M$ is a group.

\end{proof}

In what follows, we will give some characterizations concerning closed subgroups of $\R^d$.

\section{Closed subgroups of $\R^d$}\label{sec-2}

Throughout this section we describe the additive closed subgroups of $\R^d$ containing $\Z^d$. We first study closed subgroups of $\R^d$ in general.

When two groups $G_1$ and $G_2$ are isomorphic we will write  $G_1\approx G_2$.
Here and subsequently all the  vector subspaces will be  real.

\subsection{General case}\noindent

We will state in this section, some basic definitions and properties of closed subgroups of $\R^d$,
for a detailed treatment and proofs we refer the reader to \cite{Bou74}.

\begin{definition}
Given $M$ a subgroup of $\R^d$, the \textit{range} of M, denoted by $\textbf{r}(M)$, is the dimension
of the subspace generated by $M$ as a real vector space.
\end{definition}

It is known that every closed subgroup of $\R^d$ is either discrete or contains a subspace of at least dimension one
(see \cite[Proposition 3]{Bou74}).

\begin{definition}
Given $M$ a closed subgroup of $\R^d$, there exists a subspace $V$ whose dimension is the largest of the dimensions of all the subspaces contained in $M$.
We will denote by $\mathbf{d}(M)$ the dimension of $V$. Note that  $\mathbf{d}(M)$ can be zero.
\label{mayor-subepacio}
\end{definition}

Note that $0\leq\mathbf{d}(M)\leq\textbf{r}(M)\leq d$.

The next theorem establishes that
every closed subgroup of $\R^d$ is the direct sum of a subspace and a discrete group.

\begin{theorem}

Let $M$ be a closed subgroup of $\R^d$ such that $\textnormal{\textbf{r}}(M)=r$ and $\mathbf{d}(M)=p$. Let $V$ be the subspace
contained in $M$ as in Definition \ref{mayor-subepacio}. There exists a basis $\{u_1,\ldots,u_d\}$ for  $\R^d$ such that $\{u_1,\ldots,u_r\}\subset M$
and $\{u_1,\ldots,u_p\}$ is a basis for $V$.
Furthermore,
$$M=\Big\{\sum_{i=1}^p t_i u_i+\sum_{j=p+1}^r n_j u_j\,:\,t_i\in\R,\,n_j\in\Z\Big\}.$$

\end{theorem}

\begin{corollary}

If $M$ is a closed subgroup of $\R^d$ such that $\textnormal{\textbf{r}}(M)=r$ and $\mathbf{d}(M)=p$, then
$$M\approx\R^p\times\Z^{r-p}.$$
\label{tl}
\end{corollary}

\subsection{Closed subgroups of $\R^d$ containing $\Z^d$}
\noindent

We are interested in  closed subgroups of $\R^d$ containing $\Z^d$.
For their understanding, the  notion of dual group is important.

\begin{definition}
Let $M$ be a subgroup of $\R^d$. Consider the set
$$M^*:=\{x\in \R^d\,:\, \langle x,m\rangle\in\Z\quad\forall\,m\in M\}.$$
Then $M^*$ is a subgroup of $\R^d$ called the \textit{dual group} of $M$.
In particular, $(\Z^d)^*=\Z^d$.
\end{definition}

Now we will  list some properties of the dual group.

\begin{proposition}
Let $M,N$ be subgroups of $\R^d$.
\begin{enumerate}
\item [i)] $M^*$ is a closed subgroup of $\R^d$.
\item [ii)]If $N\subset M$, then
$M^*\subset N^*$.
\item [iii)] If $M$ is closed, then $\textnormal{\textbf{r}}(M^*)=d-\mathbf{d}(M)$ and $\mathbf{d}(M^*)=d-\textnormal{\textbf{r}}(M)$.
\item [iv)] $(M^*)^*=\overline{M}$.
\end{enumerate}
\label{propM*}
\end{proposition}

Let $H$ be a subgroup of $\Z^d$ with $\textnormal{\textbf{r}}(H)=q$, we will say that a set $\{v_1,\ldots,v_q\}\subset H$ is a
\textit{basis} for $H$
if for every  $x\in H$ there exist unique $k_1,\ldots,k_q\in\Z$ such that $$x=\sum_{i=1}^q k_iv_i.$$

Note that   $\{v_1,\ldots,v_d\}\subset \Z^d$  is a basis for $\Z^d$ if and only if the determinant of the
matrix $A$ which has  $\{v_1,\ldots,v_d\}$ as columns is 1 or $-1$.

Given  $B=\{v_1,\ldots,v_d\}$  a basis for $\Z^d$, we will call $\widetilde{B}=\{w_1,\ldots,w_d\}$ a {\it dual basis}
for  $B$ if $\langle v_i,w_j\rangle=\delta_{i,j}$ for all $1\leq i,j\leq d$.

If we denote by $\widetilde{A}$ the matrix with columns $\{w_1,\ldots,w_d\}$, the relation between
$B$ and $\widetilde{B}$ can be expressed in terms of matrices as $\widetilde{A}=(A^*)^{-1}$.

The closed subgroups $M$ of $\R^d$ containing $\Z^d$, can be described with the help of the dual relations. Since $\Z^d\subset M$, we have that $M^*\subset\Z^d$. So, we need first the characterization of the subgroups of $\Z^d$. This is stated in the following theorem.

\begin{theorem}\label{teogrz}
Let $H$ be a subgroup of $\Z^d$ with $\textnormal{\textbf{r}}(H)=q$, then there exist a basis $\{w_1,\ldots,w_d\}$ for $\Z^d$
and  unique integers $a_1,\ldots,a_q$ satisfying
$a_{i+1}\equiv 0\,(\text{mod. }a_i)$ for all $1\leq i\leq q-1$,
such that $\{a_1w_1,\ldots,a_q w_q\}$ is a basis for $H$.
The integers $a_1,\ldots,a_q$ are called \textnormal{invariant factors}.
\end{theorem}

The proof of the previous result can be found in \cite{Bou81}. 

\begin{remark}
Under the assumptions of the above theorem we obtain
$$\Z^d/H\approx \Z_{a_1}\times\ldots\times\Z_{a_q}\times \Z^{d-q}.$$
\end{remark}

We are now able to characterize the closed subgroups of $\R^d$ containing $\Z^d$. The proof of the following theorem can be found in \cite{Bou74}, but we include it here for the sake of completeness.

\begin{theorem}\label{como-es-M}
Let $M\subset \R^d$. The following conditions are equivalent:
\begin{enumerate}
\item [i)] $M$ is a closed subgroup of $\R^d$ containing $\Z^d$ and $\mathbf{d}(M)=d-q$.
\item [ii)]There exist a basis $\{v_1,\ldots,v_d\}$ for $\Z^d$
and integers $a_1,\ldots,a_q$ satisfying
$a_{i+1}\equiv 0\,(\text{mod. }a_i)$ for all $1\leq i\leq q-1$, such that
$$M=\Big\{\sum_{i=1}^q k_i\frac{1}{a_i}v_i+\sum_{j=q+1}^d t_j v_j\,:\,k_i\in\Z, t_j\in\R\Big\}.$$
\end{enumerate}
Furthermore, the integers $q$ and $a_1,\ldots,a_q$ are uniquely determined by $M$.
\end{theorem}

\begin{proof}

Suppose i) is true. Since $\Z^d\subset M$ and $\mathbf{d}(M)=d-q$, we have that $M^*\subset\Z^d$ and $\textnormal{\textbf{r}}(M^*)=q$.
By Theorem \ref{teogrz}, there exist  invariant factors $a_1,\ldots,a_q$ and $\{w_1,\ldots,w_d\}$ a basis for $\Z^d$ such that $\{a_1w_1,\ldots,a_q w_q\}$ is a basis for $M^*$.

Let $\{v_1,\ldots,v_d\}$ be the dual basis for
$\{w_1,\ldots,w_d\}$.

Since $M$ is closed, it follows from item iv) of Proposition \ref{propM*} that $M=(M^*)^*$.
So, $m\in M$ if and only if
\begin{equation}\label{ec1}
\langle m, a_jw_j\rangle\in\Z\quad\forall\,\,1\leq j\leq q.
\end{equation}
As $\{v_1,\ldots,v_d\}$ is a basis, given $u\in\R^d$, there exist $u_i\in\R$ such that $u=\sum_{i=1}^d u_i v_i$.
Thus, by (\ref{ec1}), $u\in M$ if and only if $u_i a_i\in\Z$ for all $1\leq i\leq q$.

We finally obtain that $u\in M$ if and only if there exist $k_i\in\Z$ and $u_j\in\R$ such that
$$u=\sum_{i=1}^q k_i\frac{1}{a_i}v_i+\sum_{j=q+1}^d u_j v_j.$$
The proof of the other implication is straightforward.

The integers $q$ and $a_1,\ldots,a_q$ are uniquely determined by $M$ since $q=d-\mathbf{d}(M)$ and
$a_1,\ldots,a_q$ are the invariant factors of $M^*$.

\end{proof}

As a consequence of the proof given above  we obtain the following coro\-llary.

\begin{corollary}\label{como-es-M*}

Let $\Z^d\subset M\subset \R^d$ be a closed subgroup with $\mathbf{d}(M)=d-q$. If $\{v_1,\ldots,v_d\}$ and  $a_1,\ldots,a_q$ are as in Theorem \ref{como-es-M}, then
$$M^*=\Big\{\sum_{i=1}^q n_i a_iw_i\,:\,n_i\in\Z\Big\},$$
where $\{w_1,\ldots,w_d\}$ is the dual basis of $\{v_1,\ldots,v_d\}$.

\end{corollary}

\begin{example}\label{ej1}

Assume that $d=3$. If $M=\frac{1}{2}\Z\times\frac{1}{3}\Z\times\R$, then  $v_1=(1,1,0)$, $v_2=(3,2,0)$ and $v_3=(0,0,1)$
verify the conditions of Theorem \ref{como-es-M} with the invariant factors $a_1=1$ and $a_2=6$.
On the other hand
$v_1'=(1,1,0)$, $v_2'=(3,2,1)$ and $v_3'=(0,0,1)$ verify the same conditions.
This shows that the basis in Theorem \ref{como-es-M} is not unique.

\end{example}

\begin{remark}\label{TL}
If $\{v_1,\ldots,v_d\}$ and  $a_1,\ldots,a_q$ are as in Theorem \ref{como-es-M}, let us define the linear transformation $T$ as
$$T:\R^d\rightarrow\R^d,\quad T(e_i)=v_i\quad\forall\, 1\leq i\leq d.$$

Then $T$ is an invertible transformation that satisfies
$$M=T\big(\frac{1}{a_1}\Z\times\dots\times\frac{1}{a_q}\Z\times\R^{d-q}\big).$$
If $\{w_1,\ldots,w_d\}$ is the dual basis for $\{v_1,\ldots,v_d\}$, the inverse of the adjoint of $T$ is defined by
$$(T^*)^{-1}:\R^d\rightarrow\R^d, \quad (T^*)^{-1}(e_i)=w_i\quad\forall\, 1\leq i\leq d.$$
By Corollary \ref{como-es-M*}, it is true that
$$M^*=(T^*)^{-1}(a_1\Z\times\dots\times a_q\Z\times\{0\}^{d-q}).$$
\end{remark}

\section{The structure of principal M-invariant spaces}\label{sec-3}

Throughout this section $M$ will be a closed subgroup of $\R^d$ containing $\Z^d$ and $M^*$ its dual group
defined as in the previous section.

\begin{definition}
We will say that a closed subspace $S$ of $L^2(\R^d)$ is {\it $M$-invariant}
if $t_mf\in S$ for all $m\in M$ and $f\in S$.

Given $\Phi\subset L^2(\R^d)$, the $M$-invariant space generated by $\Phi$ is
$$S_M(\Phi)=\sn(\{t_m\varphi\,:\,m\in M\,\,,\,\,\varphi\in\Phi\}).$$
If $\Phi=\{\varphi\}$ we write $S_M(\varphi)$ and we say that $S_M(\varphi)$ is a principal $M$-invariant space.
For simplicity of notation, when $M=\Z^d$, we write $S(\varphi)$ instead of $S_{\Z^d}(\varphi)$.
\end{definition}

Principal SISs have been completely characterized by \cite{BDR94} (see also \cite{BDVR},\cite{RS}) as follows.

\begin{theorem}
Let $f\in L^2(\R^d)$ be given. If $g\in S(f)$, then there exists a $\Z^d$-periodic function
$\eta$ such that $\widehat{g}=\eta\widehat{f}$.

Conversely, if $\eta$ is a $\Z^d$-periodic function such that $\eta\widehat{f}\in L^2(\R^d)$, then the
function $g$ defined by $\widehat{g}=\eta\widehat{f}$ belongs to $S(f)$.
\label{rango-SIS-M}
\end{theorem}

The aim of this section is to generalize the previous theorem to the $M$-invariant case.
In case that $M$ is discrete, Theorem \ref{rango-SIS-M} follows easily by rescaling.
The difficulty arises when $M$
is not discrete.

\begin{theorem}\label{cha-PSIS}
Let $f\in L^2(\R^d)$ and $M$ a closed subgroup of $\R^d$ containing $\Z^d$.
If $g\in S_M(f)$, then there exists an $M^*$-periodic function
$\eta$ such that $\widehat{g}=\eta\widehat{f}$.

Conversely, if $\eta$ is an $M^*$-periodic function such that $\eta\widehat{f}\in L^2(\R^d)$, then the
function $g$ defined by $\widehat{g}=\eta\widehat{f}$ belongs to $S_M(f)$.
\label{rango-SIS-M-2}
\end{theorem}

Theorem \ref{cha-PSIS} was proved in \cite{BDR94} for the lattice case. 
We adapt their arguments to this more general case.

We will first need some definitions and properties.

By Remark \ref{TL}, there exists a linear transformation $T:\R^d\rightarrow\R^d$ such that
$M=T\big(\frac{1}{a_1}\Z\times\dots\times\frac{1}{a_q}\Z\times\R^{d-q}\big)$ and $M^*=(T^*)^{-1}(a_1\Z\times\dots\times a_q\Z\times\{0\}^{d-q})$, where $q=d-\textbf{d}(M)$.

We will denote by $\SR$ the section of the quotient  $\R^d/M^*$ defined as
\begin{equation}\label{R-def}
\SR=(T^*)^{-1}([0,a_1)\times\dots\times[0,a_q)\times\R^{d-q}).
\end{equation}
Therefore, $\{\SR+m^*\}_{m^*\in M^*}$ forms a partition of $\R^d$.

Given $f,g \in L^2(\R^d)$ we define
$$[f,g](\w):=\sum_{m^*\in M^*}f(\w+m^*)\overline{g(\w+m^*)},$$
where $\w \in \SR$. Note that, as $f,g\in L^2(\R^d)$ we have that $[f,g]\in L^1(\SR)$, since
\begin{align}\label{cuenta}
\int_{\R^d}f(\w)\overline{g(\w)}\,d\w&=\sum_{m^*\in M^*}\int_{\SR+m^*}f(\w)\overline{g(\w)}\,d\w\notag\\
&=\sum_{m^*\in M^*}\int_{\SR}f(\w+m^*)\overline{g(\w+m^*)}\,d\w\notag\\
&=\int_{\SR}[f,g](\w)\,d\w.
\end{align}
From this, it follows  that if $f\in L^2(\R^d)$, then $\{f(\w+m^*)\}_{m^*\in M^*}\in \ell^2(M^*)$ a.e.
$\w\in\SR$.

The Cauchy-Schwarz inequality in $\ell^2(M^*)$, gives the following a.e. pointwise estimate
\begin{equation}
|[f,g]|^2\leq [f,f][g,g]
\label{C-S}
\end{equation}
for every $f,g\in L^2(\R^d)$.

Given an  $M^*$-periodic function $\eta$ and $f,g\in L^2(\R^d)$ such that $\eta f\in L^2(\R^d)$,
it is easy to check that
\begin{equation}
[\eta f,g]=\eta[f,g].
\label{eta-sale}
\end{equation}

The following lemma is an  extension to  general subgroups of $\R^d$ of a result which holds for the discrete case.

\begin{lemma}
Let $f\in L^2(\R^d)$, $M$ a closed subgroup of $\R^d$ containing $\Z^d$ and $\SR$ defined as in (\ref{R-def}). Then,
$$S_{M}(f)^{\perp}=\{g\in L^2(\R^d)\,: \,[\widehat{f},\widehat{g}](\w)=0\,\,\textrm{a.e.}\,\,\w\in\SR\}.$$
\label{lema-ort}
\end{lemma}

\begin{proof}
Since the span of the set $\{t_mf\,:\,m\in M\}$ is dense in $S_{M}(f)$, we have that
$g\in S_{M}(f)^{\perp}$ if and only if $\langle \widehat{g}, e_m\widehat{f}\rangle=0$ for all $m\in M$.
As $e_m$ is $M^*$-periodic, using  (\ref{cuenta}) and (\ref{eta-sale}), we obtain that
$g\in S_{M}(f)^{\perp}$ if and only if
\begin{equation}
\int_{\SR}e_m(\w)[\widehat{f},\widehat{g}](\w)\,d\w=0,
\label{corchete}
\end{equation}
for all $m\in M$.

At this point, what is left to show is that if (\ref{corchete}) holds then $[\widehat{f},\widehat{g}](\w)=0$ a.e. $\w\in\SR$.
For this, taking into account that $[\widehat{f},\widehat{g}]\in L^1(\SR)$, it is enough to prove that if $h\in L^1(\SR)$ and $\int_{\SR}he_m=0$
for all $m\in M$ then $h=0$ a.e $\w\in \SR$.

We will prove the preceding property for the case $M=\Z^q\times\R^{d-q}$. The general case will follow from
a change of variables using the description of $M$ and $\SR$ given in Remark \ref{TL} and (\ref{R-def}).

Suppose now  $M=\Z^q\times\R^{d-q}$, then $\SR=[0,1)^q\times\R^{n-q}$. Take $h\in L^1(\SR)$, such that
\begin{equation}\label{cuenta2}
\iint_{[0,1)^q\times\R^{n-q}}h(x,y)e^{-2\pi i (kx+ty)}\,dxdy=0\quad\forall\,k\in\Z^q,\,t\in\R^{d-q}.
\end{equation}
Given $k\in\Z^q$, define $\alpha_k(y):=\int_{[0,1)^q}h(x,y)e^{-2\pi i kx}\,dx$ for a.e. $y\in\R^{d-q}$.
It follows from (\ref{cuenta2}) that
\begin{equation}\label{cuenta4}
\int_{\R^{d-q}}\alpha_k(y)e^{-2\pi i ty}\,dy=0\quad\forall\,t\in\R^{d-q}.
\end{equation}
Since  $h\in L^1(\SR)$, by Fubini's Theorem,
$\alpha_k\in L^1([0,1)^q)$. Thus, using (\ref{cuenta4}), $\alpha_k(y)=0$ a.e. $y\in\R^{d-q}$. That is
\begin{equation}\label{cuenta3}
\int_{[0,1)^q}h(x,y)e^{-2\pi i kx}\,dx=0
\end{equation} for a.e. $y\in\R^{d-q}$.
Define now $\beta_y(x):=h(x,y)$. By (\ref{cuenta3}), for a.e. $y\in\R^{d-q}$ we have
that $\beta_y(x)=0$ for a.e. $x\in [0,1)^q$. Therefore, $h(x,y)=0$ a.e. $(x,y)\in[0,1)^q\times\R^{d-q}$
and this completes the proof.
\end{proof}

Now we will give a formula for the orthogonal projection onto $S_M(f)$.

\begin{lemma}\label{lema-proy-ort}
Let $P$ be the orthogonal projection onto $S_M(f)$. Then, for each $g\in L^2(\R^d)$, we have $\widehat{Pg}=\eta_g\widehat{f}$, where $\eta_g$ is the
$M^*$-periodic function defined by
$$\eta_g:=
\begin{cases}
[\widehat{g},\widehat{f}]/[\widehat{f},\widehat{f}]&\,\textrm{on}\, E_f+M^*\\
0\,&\,\textrm{otherwise},
\end{cases}$$
and $E_f$ is the set $\{ \w\in\SR\,:\, [\widehat{f},\widehat{f}](\w)\neq 0\}$.
\end{lemma}

\begin{proof}
Let $\widehat{P}$ be the orthogonal projection onto $\widehat{S_M(f)}$. Since $\widehat{Pg}=\widehat{P}\widehat{g}$,
it is enough to show that $\widehat{P}\widehat{g}=\eta_g\widehat{f}$.

We first want to prove that $\eta_g\widehat{f}\in L^2(\R^d)$. Combining (\ref{cuenta}), (\ref{C-S}) and
(\ref{eta-sale})
$$\int_{\R^d}|\eta_g\widehat{f}|^2=\int_{\SR}|\eta_g|^2[\widehat{f},\widehat{f}]\leq\int_{\SR}[\widehat{g},\widehat{g}]
=\|g\|_{L^2}^2,$$
and so, $\eta_g\widehat{f}\in L^2(\R^d)$. Define the linear map
$$Q:L^2(\R^d)\longrightarrow L^2(\R^d),\quad Q\widehat{g}=\eta_g\widehat{f},$$
which is well defined and has norm not greater than one. We will prove that $Q=\widehat{P}$.

Take $\widehat{g}\in\widehat{S_M(f)}^{\perp}=(S_M(f)^{\perp})^{\land}$. Then Lemma \ref{lema-ort} gives that
$\eta_g=0$, hence $Q\widehat{g}=0$. Therefore, $Q=\widehat{P}$ on $\widehat{S_M(f)}^{\perp}$.

On the other hand, on $E_f+M^*$,
$$\eta_{(t_mf)}=[e_m\widehat{f},\widehat{f}]/[\widehat{f},\widehat{f}]=e_m\quad\forall\,m\in M.$$
Since $\widehat{f}=0$ outside of $E_f+M^*$, we have that $Q(\widehat{t_mf})=e_m\widehat{f}$.
As $Q$ is linear and bounded, and the set $\text{span}\{t_mf\,:\,m\in M\}$ is dense in $S_{M}(f)$,
 $Q=\widehat{P}$ on $\widehat{S_{M}(f)}$.

\end{proof}

\begin{proof}[Proof of Theorem \ref{rango-SIS-M-2}]

Suppose that $g\in S_M(f)$, then $Pg=g$, where $P$ is the orthogonal projection onto $S_M(f)$. Hence,
by Lemma \ref{lema-proy-ort}, $\widehat{g}=\eta_g\widehat{f}$.

Conversely, if $\eta\widehat{f}\in L^2(\R^d)$ and $\eta$ is an $M^*$-periodic function,
then $g$, the inverse transform of $\eta\widehat{f}$ is also in $L^2(\R^d)$ and satisfies, by (\ref{eta-sale}),
that $\eta_g=[\eta\widehat{f},\widehat{f}]/[\widehat{f},\widehat{f}]=\eta$ on $E_f+M^*$.

On the other hand, since $\text{supp}(\widehat{f})\subset E_f+M^*$, we have that $\eta_g\widehat{f}=\eta\widehat{f}$.

So, $\widehat{Pg}=\eta_g\widehat{f}=\eta\widehat{f}=\widehat{g}$.
Consequently, $Pg=g$, and hence $g\in S_M(f)$.

\end{proof}

\section{Characterization of  M-invariance}\label{sec-4}

Given $M$ a closed subgroup of $\R^d$ containing $\Z^d$,
our goal is  to characterize when a SIS S is an $M$-invariant space.
For this, we will construct  a partition $\{B_{\s}\}_{\s\in\SN}$ of $\R^d$, where each $B_{\s}$ will be  an $M^*$-periodic set and the index set $\SN$ will be properly chosen later (see(\ref{como-es-N})).
Using this partition, for each $\s\in \SN$, we define the subspaces
\begin{equation}\label{U-sigma}
U_{\s}=\{f\in L^2(\R^d):\widehat{f}=\chi_{B_{\s}}\widehat{g},
\,\,\textrm{with} \,\,g\in S\}.
\end{equation}

The main theorem of this section characterizes the $M$-invariance of $S$ in terms of
the subspaces $U_{\s}$.

\begin{theorem}\label{teo-subs}
If $S\subseteq L^2(\R^d)$ is a SIS and $M$ is a closed subgroup of $\R^d$ containing $\Z^d$, then the following are equivalent.
\begin{enumerate}
\item [i)]$S$ is $M$-invariant.
\item [ii)]$U_{\s}\subseteq S$ for all $\s\in \SN$.
\end{enumerate}
Moreover, in case any of the above holds, we have that $S$ is the orthogonal direct sum
$$S=\dotbigoplus_{\s\in \SN} U_{\s}.$$
\end{theorem}

Before proving the theorem let us  carefully define the partition $\{B_{\s}\}_{\s\in\SN}$, in such a way that
each $B_{\s}$ is an $M^*$-periodic set.

Let $\W$ be the  section of the quotient $\R^d/\Z^d$ given by
\begin{equation}\label{como-es-W}
\W=(T^*)^{-1}([0,1)^d),
\end{equation}
 where $T$ is
as in Remark \ref{TL}.
Then $\W$ tiles $\R^d$ by $\Z^d$ translations, that is
\begin{equation}
\R^d=\bigcup_{k\in\Z^d}\W+k.
\label{0mega-tesela}
\end{equation}

Now, for each $k\in\Z^d$,  consider $(\W+k)+M^*$. Although these sets are $M^*$-periodic, they are not a partition of $\R^d$. So, we need to choose  a subset $\SN$ of $\Z^d$ such that
if $\s,\s'\in\SN$ and $\s+M^*=\s'+M^*$, then $\s=\s'$. Thus $\SN$ should be a section of the quotient $\Z^d/M^*$. We can choose for example the set $\SN$ given by
\begin{equation}\label{como-es-N}
\SN=(T^*)^{-1}(\{0,\ldots,a_1-1\}\times\ldots\times\{0,\ldots,a_q-1\}\times\Z^{d-q}),
 \end{equation}
 where $a_1,\ldots,a_q$ are the invariant factors of $M$.
Hence, given $\s\in\SN$
we define
\begin{equation}\label{B-sigma}
B_{\s}=\W+\s+M^*=\bigcup_{m^*\in M^*}(\W+\s)+m^*.
\end{equation}

We give three basic examples of this construction.

\begin{example}
\noindent

$(1)$ Let $M=\frac{1}{n}\Z\subset\R$, then $M^*=n\Z$,  $\W=[0,1)$ and $\SN=\{0,\ldots,n-1\}$.
Given $\s\in \{0,\ldots,n-1\}$, we have $$B_{\s}=\bigcup_{m^*\in n\Z} ([0,1)+\s)+m^*=
\bigcup_{j\in\Z}[\s,\s+1)+n j.$$
Figure 1 illustrates  the partition for $n=4$. In the picture, the black dots represent the set $\SN$. The set $B_2$
is the one which appears in gray.
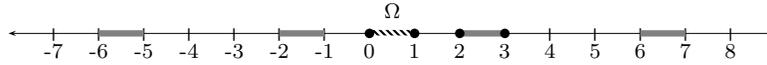
\begin{figure}[h!]
\centering
\begin{flushleft}
\psset{unit=0.6cm}
\begin{pspicture}(-10,-1.5)(2,1)
\psset{unit=0.6cm}
\psline[linewidth=0.3pt]{<->}(-8,0)(9,0)

\psdot[dotstyle=|](4,0)
\psdot[dotstyle=|](5,0)
\psdot[dotstyle=|](6,0)
\psdot[dotstyle=|](7,0)
\psdot[dotstyle=|](8,0)
\psdot[dotstyle=|](-1,0)
\psdot[dotstyle=|](-2,0)
\psdot[dotstyle=|](-3,0)
\psdot[dotstyle=|](-4,0)
\psdot[dotstyle=|](-5,0)
\psdot[dotstyle=|](-6,0)
\psdot[dotstyle=|](-7,0)

\rput(0,-0.4){\scriptsize{0}}
\rput(1,-0.4){\scriptsize{1}}
\rput(2,-0.4){\scriptsize{2}}
\rput(3,-0.4){\scriptsize{3}}
\rput(4,-0.4){\scriptsize{4}}
\rput(5,-0.4){\scriptsize{5}}
\rput(6,-0.4){\scriptsize{6}}
\rput(7,-0.4){\scriptsize{7}}
\rput(8,-0.4){\scriptsize{8}}
\rput(-1,-0.4){\scriptsize{-1}}
\rput(-2,-0.4){\scriptsize{-2}}
\rput(-3,-0.4){\scriptsize{-3}}
\rput(-4,-0.4){\scriptsize{-4}}
\rput(-5,-0.4){\scriptsize{-5}}
\rput(-6,-0.4){\scriptsize{-6}}
\rput(-7,-0.4){\scriptsize{-7}}
\rput(0.5,0.5){\scriptsize{$\Omega$}}
\psframe[linecolor=gray, linestyle=dotted, dotsep=500pt, fillstyle=vlines*,
hatchsep=1pt](0,-0.07)(1,0.1)
\psframe*[linecolor=gray](2,-0.07)(3,0.07)
\psframe*[linecolor=gray](6,-0.07)(7,0.07)
\psframe*[linecolor=gray](-2,-0.07)(-1,0.07)
\psframe*[linecolor=gray](-6,-0.07)(-5,0.07)
\psdot[dotstyle=*](0,0)
\psdot[dotstyle=*](1,0)
\psdot[dotstyle=*](2,0)
\psdot[dotstyle=*](3,0)
\end{pspicture}
\end{flushleft}
\caption{\small Partition of the real line for $M=\frac1{4}\Z$.}
\end{figure}

$(2)$ Let $M=\frac1{2}\mathbb{Z}\times\mathbb{R}$, then $\Omega=[0,1)^2$, $M^*=2\Z\times\{0\}$ and
$\mathcal{N}=\{0,1\}\times\mathbb{Z}$.

So, the sets $B_{(i,j)}$ are
$$B_{(i,j)}=\bigcup_{k\in\mathbb{Z}}\big([0,1)^2+(i,j)\big)+(2k,0)$$
where $(i,j)\in\mathcal{N}$. See Figure 2, where the sets $B_{(0,0)}$, $B_{(1,1)}$ and  $B_{(1,-1)}$ are represented by the squares painted in light gray,
gray and dark gray respectively. As in the previous figure, the set $\SN$ is represented by the black dots.
\begin{figure}[h!]
\centering
\begin{center}
\psset{unit=0.7cm}
\begin{pspicture}(-9,-3)(9,4)
\psgrid[subgriddiv=0,griddots=5,gridlabels=7pt,gridcolor=lightgray](-7,-2)(7,3)

\psline[linewidth=1pt]{<->}(0,-2)(0,3)
\rput(-0.3,-0.3){$\scriptsize{o}$}

\psframe*[linecolor=lightgray](0,0)(1,1)
\rput(0.5,0.5){$\Omega$}
\psframe*[linecolor=lightgray](2,0)(3,1)
\psframe*[linecolor=lightgray](4,0)(5,1)
\psframe*[linecolor=lightgray](-2,0)(-1,1)
\psframe*[linecolor=lightgray](-4,0)(-3,1)
\psframe*[linecolor=lightgray](6,0)(7,1)
\psframe*[linecolor=lightgray](-6,0)(-5,1)

\psframe*[linecolor=gray](1,1)(2,2)
\psframe*[linecolor=gray](3,1)(4,2)
\psframe*[linecolor=gray](5,1)(6,2)
\psframe*[linecolor=gray](-1,1)(0,2)
\psframe*[linecolor=gray](-3,1)(-2,2)
\psframe*[linecolor=gray](-5,1)(-4,2)
\psframe*[linecolor=gray](-7,1)(-6,2)

\psframe*[linecolor=darkgray](1,-1)(2,0)
\psframe*[linecolor=darkgray](3,-1)(4,0)
\psframe*[linecolor=darkgray](5,-1)(6,0)
\psframe*[linecolor=darkgray](-1,-1)(0,0)
\psframe*[linecolor=darkgray](-3,-1)(-2,0)
\psframe*[linecolor=darkgray](-5,-1)(-4,0)
\psframe*[linecolor=darkgray](-7,-1)(-6,0)
\psline[linewidth=1pt]{<->}(-7,0)(7,0)
\psdot[dotstyle=*](0,0)
\psdot[dotstyle=*](0,1)
\psdot[dotstyle=*](1,0)
\psdot[dotstyle=*](1,1)
\psdot[dotstyle=*](0,2)
\psdot[dotstyle=*](1,2)
\psdot[dotstyle=*](0,3)
\psdot[dotstyle=*](1,3)
\psdot[dotstyle=*](0,-1)
\psdot[dotstyle=*](1,-1)
\psdot[dotstyle=*](0,-2)
\psdot[dotstyle=*](1,-2)

\end{pspicture}
\end{center}
\caption{\small Partition of the plane  for $M=\frac1{2}\mathbb{Z}\times\mathbb{R}$.}
\end{figure}
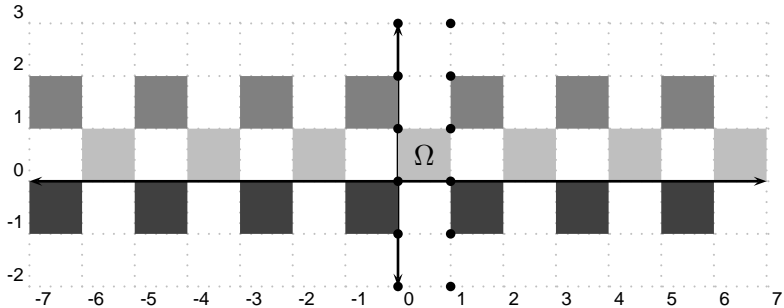

$(3)$ Let $M=\{k\frac1{3}v_1+tv_2\,:\, k\in \mathbb{Z}\,\,\textrm{and}\,\,t\in\mathbb{R}\}$, where $v_1=(1,0)$ and $v_2=(-1,1)$. Then, $\{v_1,v_2\}$ satisfy conditions in Theorem \ref{como-es-M}. By Corollary \ref{como-es-M*},
$M^*=\{k3 w_1\,:\,\,k\in\mathbb{Z}\}$, where $w_1=(1,1)$ and $w_2=(0,1)$.

Note that the sets $\W$ and $\SN$ can be expressed in terms of $w_1$ and $w_2$ as
$$\Omega=\{tw_1+sw_2:\,\, t,s\in [0,1)\}$$ and
$$\mathcal{N}=\{aw_1+kw_2:\,\, a\in\{0,1,2\},\,\,k\in\mathbb{Z}\}.$$
This is illustrated in Figure 3. In this case the sets $B_{(0,0)}$, $B_{(1,0)}$ and  $B_{(1,2)}$ correspond to the light gray, gray and dark gray regions respectively. And again, the black dots represent the set $\SN$.
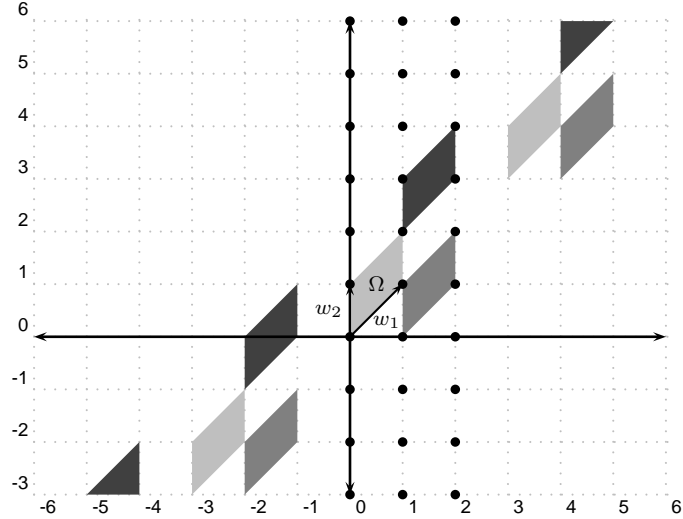
\begin{figure}[h!]
\centering
\begin{center}
\psset{unit=0.7cm}
\begin{pspicture}(-7,-4)(7,7)
\psgrid[subgriddiv=0,griddots=5,gridlabels=7pt,gridcolor=lightgray](-6,-3)(6,6)
\psline[linewidth=1pt]{<->}(0,-3)(0,6)

\pspolygon*[linecolor=lightgray](0,0)(0,1)(1,2)(1,1)
\rput(0.5,1){\scriptsize{$\Omega$}}
\pspolygon*[linecolor=lightgray](3,3)(3,4)(4,5)(4,4)
\pspolygon*[linecolor=lightgray](-2,-1)(-2,-2)(-3,-3)(-3,-2)
\pspolygon*[linecolor=gray](1,0)(1,1)(2,2)(2,1)
\pspolygon*[linecolor=gray](4,3)(4,4)(5,5)(5,4)
\pspolygon*[linecolor=gray](-1,-1)(-1,-2)(-2,-3)(-2,-2)
\pspolygon*[linecolor=darkgray](1,2)(1,3)(2,4)(2,3)
\pspolygon*[linecolor=darkgray](4,5)(4,6)(5,6)
\pspolygon*[linecolor=darkgray](-1,1)(-1,0)(-2,-1)(-2,0)
\pspolygon*[linecolor=darkgray](-4,-2)(-4,-3)(-5,-3)
\psline[linewidth=1pt]{<->}(-6,0)(6,0)
\psline[linewidth=0.8pt]{->}(0,0)(1,1)
\psline[linewidth=0.8pt]{->}(0,0)(0,1)
\rput[t] (-0.4,0.6){\scriptsize{$w_2$}}
\rput[t] (0.7,0.4){\scriptsize{$w_1$}}
\psdot[dotstyle=*](0,0)
\psdot[dotstyle=*](1,0)
\psdot[dotstyle=*](0,1)
\psdot[dotstyle=*](2,0)
\psdot[dotstyle=*](2,1)
\psdot[dotstyle=*](0,2)
\psdot[dotstyle=*](0,3)
\psdot[dotstyle=*](0,4)
\psdot[dotstyle=*](0,5)
\psdot[dotstyle=*](0,6)
\psdot[dotstyle=*](0,-1)
\psdot[dotstyle=*](0,-2)
\psdot[dotstyle=*](0,-3)
\psdot[dotstyle=*](1,1)
\psdot[dotstyle=*](1,2)
\psdot[dotstyle=*](1,3)
\psdot[dotstyle=*](1,4)
\psdot[dotstyle=*](1,5)
\psdot[dotstyle=*](1,6)
\psdot[dotstyle=*](1,-1)
\psdot[dotstyle=*](1,-2)
\psdot[dotstyle=*](1,-3)
\psdot[dotstyle=*](2,2)
\psdot[dotstyle=*](2,3)
\psdot[dotstyle=*](2,4)
\psdot[dotstyle=*](2,5)
\psdot[dotstyle=*](2,6)
\psdot[dotstyle=*](2,-1)
\psdot[dotstyle=*](2,-2)
\psdot[dotstyle=*](2,-3)

\end{pspicture}
\end{center}
\label{ejemplo-torcido}
\caption{\small Partition for $M=\{k\frac1{3}(1,0)+t(-1,1)\,:\, k\in \mathbb{Z}\,\,\textrm{and}\,\,t\in\mathbb{R}\}$.}
\end{figure}
\end{example}

Once the partition $\{B_{\s}\}_{\s\in\SN}$ is defined, we state a lemma which will be necessary to prove Theorem \ref{teo-subs}.

\begin{lemma}\label{U}
Let $S$ be a SIS and  $\s\in\SN$. Assume that the subspace $U_{\s}$ defined in (\ref{U-sigma}) satisfies $U_{\s}\subseteq S$. Then, $U_{\s}$ is an $M$-invariant space and in particular is a SIS.
\end{lemma}

\begin{proof}
Let us first prove that $U_{\s}$ is closed.
Suppose that $f_j\in U_{\s}$ and $f_j\rightarrow f$ in $L^2(\R^d)$.
Since $U_{\s}\subseteq S$  and $S$ is closed, $f$ must be in $S$.
Further,
$$\|\widehat{f_j}-\widehat{f}\|^2_2=\|(\widehat{f_j}-\widehat{f})\chi_{B_{\s}}\|^2_2+\|(\widehat{f_j}-\widehat{f})\chi_{B_{\s}^c}\|^2_2
=\|\widehat{f_j}-\widehat{f}\chi_{B_{\s}}\|^2_2+\|\widehat{f}\chi_{B_{\s}^c}\|^2_2.$$

Since the left-hand side converges to zero, we must have that $\widehat{f}\chi_{B_{\s}^c}=0$ a.e.
$\w\in\R^d$, and $\widehat{f_j}\rightarrow \widehat{f}\chi_{B_{\s}}$ in $L^2(\R^d)$. So, as
$\widehat{f_j}\rightarrow \widehat{f}$ in $L^2(\R^d)$, we conclude that
$$\widehat{f}=\widehat{f}\chi_{B_{\s}},$$
which proves that $f\in U_{\s}$ and so $U_{\s}$ is closed.

Note that, since $\Z^d\subset M$, the $\Z^d$-invariance of $U_{\s}$ is a consequence of the $M$-invariance.

So, given $m\in M$ and $f\in U_{\s}$, we will prove that
$e_m\widehat{f}\in \widehat{U}_{\s}.$
Since $f\in U_{\s}$,
there exists $g\in S$ such that $\widehat{f}=\chi_{B_{\s}}\widehat{g}$.
Hence,
\begin{equation}\label{ecu-1-lema-U}
e_m\widehat{f}=e_m(\chi_{B_{\s}}\widehat{g})=\chi_{B_{\s}}(e_m\widehat{g}).
\end{equation}
If we can find a $\Z^d$-periodic function $\ell_m$ verifying
\begin{equation}\label{ecu-2-lema-U}
e_m(\w)=\ell_m(\w)\quad\textrm{a.e.}\,\w\in B_{\s},
\end{equation}
then, we can rewrite (\ref{ecu-1-lema-U}) as
$$e_m\widehat{f}=\chi_{B_{\s}}(\ell_m\widehat{g}).$$
By Theorem \ref{rango-SIS-M},  $\ell_m\widehat{g}\in \widehat{S(g)}\subset\widehat{S}$
and so, $e_m\widehat{f}\in \widehat{U}_{\s}.$

Let us now define the function $\ell_m$.
Note that, since $e_m$ is $M^*$-periodic,
\begin{equation}\label{ecu-3-lema-U}
e_m(\w+\s)=e_m(\w+\s+m^*)\quad\textrm{a.e.}\,\,\w\in\W,\, \forall\,m^*\in M^*.
\end{equation}
For each $k\in\Z^d$, set
\begin{equation}\label{ecu-4-lema-U}
\ell_m(\w+k)=e_m(\w+\s)\quad\textrm{a.e.}\,\,\w\in\W.
\end{equation}
It is clear that $\ell_m$ is $\Z^d$-periodic and combining (\ref{ecu-3-lema-U}) with (\ref{ecu-4-lema-U}), we obtain (\ref{ecu-2-lema-U}).

\end{proof}

\begin{proof}[Proof of Theorem \ref{teo-subs}]
i)\,$\Rightarrow$\,ii): Fix $\s\in \SN$ and $f\in U_{\s}$. Then $\widehat{f}=\chi_{B_{\s}}\widehat{g}$
for some $g\in S.$ Since $\chi_{B_{\s}}$ is an $M^*$-periodic function, by Theorem \ref{rango-SIS-M-2}, we have that $f\in S_M(g)\subseteq S$, as we wanted to prove.

ii)\,$\Rightarrow$\,i): Suppose that $U_{\s}\subseteq S$ for all $\s\in \SN$.
Note that Lemma \ref{U} implies that
$U_{\s}$ is $M$-invariant, and we also have that the subspaces $U_{\s}$ are mutually orthogonal
since the sets $B_{\s}$ are disjoint.

Take $f\in S$. Then, since $\{B_{\s}\}_{\s\in \SN}$ is a partition of $\R^d$, we can decompose $f$
as $f=\sum_{\s\in \SN}f^{\s}$ where $f^{\s}$ is such that $\widehat{f^{\s}}=\widehat{f}\chi_{B_{\s}}$.
This implies that $f\in \dotbigoplus_{\s\in \SN} U_{\s}$ and
consequently, $S$ is the orthogonal direct sum
$$S=\dotbigoplus_{\s\in \SN} U_{\s}.$$
As each $U_{\s}$ is $M$-invariant, so is $S$.
\end{proof}

\subsection{Characterization of $M$-invariance in terms of fibers}\noindent

A useful tool in the theory of shift-invariant spaces is based on
early work of Helson \cite{Hel64}.
An $L^2(\R^d)$ function is decomposed into ``fibers''.
This produces a characterization of SIS in terms of
closed subspaces of $\ell^2(\Z^d)$ (the fiber spaces).

For a detailed description of this approach, see \cite{Bow00} and the
references therein.

Given $f \in L^2(\R^d)$ and $\omega \in \W$,
the \emph{fiber} $\widehat{f}_\omega$ of~$f$ at~$\omega$ is the sequence
$$\widehat{f}_\omega = \{\widehat{f}(\omega+k)\}_{k \in \Z^d}.$$

If $f$ is in $L^2(\R^d)$, then the fiber $\widehat{f}_\omega$ belongs
to $\ell^2(\Z^d)$ for almost every $\omega \in \W$.

Given a  subspace $V$ of $L^2(\R^d)$ and $\omega \in \W$,
the \emph{fiber space} of~$V$ at~$\omega$ is
$$J_V(\omega)
=\overline{\{\widehat{f}_\omega :
    f \in V }\},$$
where the closure is taken in the norm of $\ell^2(\Z^d)$.

The map assigning to each $\omega$ the fiber space $J_V(\omega)$
is known in the literature as the {\it range function} of $V$.

The {\it dimension function} is defined by
$$\dim_V:\W\rightarrow \N_0\cup\{\infty\}, \quad \dim_V(\w):=\dim(J_V(\w)).$$

For a proof that, for almost every~$\omega$, $J_V(\omega)$ is a well-defined closed subspace
of $\ell^2(\Z^d)$  and that shift-invariant spaces can be characterized through range functions,
see \cite{Bow00}, \cite{Hel64}.

\begin{proposition}[\cite{Hel64}] \label{bownik-1}
If $S$ is a SIS, then
$$S \EQ \bigset{f \in L^2(\R^d) :
        \fhat_\omega \in J_S(\omega) \text{ for a.e. }\w\in\W }.$$
\end{proposition}

\begin{proposition} \label{ortho}
Let $S_1$ and $S_2$ be SISs. Then we have,
 $S = S_1 \dotoplus S_2$, if and only if
$$J_S(\omega)
\EQ J_{S_1}(\omega) \dotoplus J_{S_2}(\omega), \quad\text{a.e. }\w\in\W.$$
\end{proposition}

Let $\Phi= \set{\varphi_1,\dots,\varphi_{\ell}}$ be a finite collection of functions in $L^2(\R^d)$.
Then the \emph{Gramian} $G_\Phi$ of $\Phi$ is the
$\ell \times \ell$ matrix of $\Z^d$-periodic functions
\begin{equation} \label{gram}
[G_{\Phi}(\omega)]_{ij}
=
\Big\langle (\widehat{\varphi_i})_{\omega}, (\widehat{\varphi_j})_{\omega}\Big\rangle
= \sum_{k \in \Z^d} \widehat{\varphi_i}(\omega+k) \, \overline{\widehat{\varphi_j}(\omega+k)},
\qquad \omega\in \W.
\end{equation}

So, we have the following relation whose proof is straightforward.

\begin{proposition}\label{rango-gramiano}
Let $S$ be an FSIS generated by $\Phi$. Therefore,
$$\dim_S(\w)=\textnormal{rank}[G_{\Phi}(\w)] \quad\text{ a.e. }\w\in \W.$$
\end{proposition}

Now, if $f\in L^2(\R^d)$ and $\s\in \SN$, let $f^{\s}$ denote the function defined by
$$\widehat{f^{\s}}=\widehat{f}\chi_{B_{\s}}.$$

Let $P_{\s}$ be the orthogonal projection onto $S_{\s}$, where
$$S_{\s}:=\{f\in L^2(\R^d)\,:\,\text{ supp}(\widehat{f})\subset B_{\s}\}.$$
Therefore
$$f^{\s}=P_{\s}f\quad\text{and}\quad U_{\s}=P_{\s}(S)=\{f^{\s}\,:\,f\in S\}.$$
Moreover, if $S=S(\Phi)$ with $\Phi$ a countable subset of $L^2(\R^d)$, then
\begin{equation}\label{JUsigma}
J_{U_{\s}}(\w)=\overline{\textnormal{span}}\{(\widehat{\varphi^{\s}})_{\w}\,:\,\varphi\in\Phi\}.
\end{equation}

\begin{remark}\label{fibra-cut-off}

Note that $(\widehat{\varphi^{\s}})_{\w}=\{\chi_{B_{\s}}(\w+k)\widehat{\varphi}(\w+k)\}_{k\in\Z^d}.$
\medskip
Then, since $\chi_{B_{\s}}(\w+k)\neq0$ if and only if $k\in\s+M^*$,
\begin{equation*}
\chi_{B_{\s}}(\w+k)\widehat{\varphi}(\w+k)=
\begin{cases}
\widehat{\varphi}(\w+k)&\text{ if }k\in\s+M^*\\
0&\text{ otherwise }.
\end{cases}
\end{equation*}
Therefore, if $\s\neq\s'$, $J_{U_{\s}}(\w)$ is orthogonal to $J_{U_{\s'}}(\w)$ for a.e. $\w\in\W$.
\end{remark}

Combining Theorem~\ref{teo-subs} with Proposition~\ref{bownik-1} and (\ref{JUsigma})
we obtain the following proposition.
\begin{theorem}\label{fibras-en-S}
Let $S$ be a SIS generated by $\Phi$. The following statements are equivalent.
\begin{enumerate}
\item[(i)]
$S$ is $M$-invariant.
\item[(ii)]
$(\widehat{\varphi^{\s}})_{\w}\in J_S(\w)$ a.e. $\w\in\W$
for all  $\varphi\in\Phi$ and $\s\in\SN$.
\end{enumerate}
\end{theorem}

Now, using Proposition \ref{ortho}, Proposition \ref{rango-gramiano}, Theorem \ref{fibras-en-S} and 
Remark \ref{fibra-cut-off}, we give a slightly simpler
characterization of $M$-invariance for the finitely generated case.

\begin{theorem} \label{fibers}
If $S$ is an FSIS generated by $\Phi$, then
the following statements are equivalent.

\begin{enumerate}
\item[(a)]
$S$ is $M$-invariant.

\vspace{.25cm}
\item[(b)]
For almost every $\omega \in \W$,\,
$\dim_S(\omega)
= \sum_{\s\in\SN} \dim_{U_{\s}}(\omega).$
\vspace{.25cm}
\item[(c)]
For almost every $\omega \in \W$,\,
$\textnormal{rank}[G_{\Phi}(\w)]=\sum_{\s\in\SN}\textnormal{rank}[G_{\Phi^{\s}}(\w)],$
\vspace{.1cm}

where $\Phi^{\s}=\{\varphi^{\s}\,:\,\varphi\in\Phi\}.$
\end{enumerate}
\end{theorem}

\section{Applications of $M$-invariance}\label{sec-5}

In this section we present two applications of the  results given before.
First,  we will
estimate  the size of the supports of the Fourier transforms of the generators
of an FSIS which is also $M$-invariant.
Finally, given $M$ a closed subgroup of $\R^d$ containing $\Z^d$, we will construct a SIS $S$ which is exactly $M$-invariant. That is, $S$ is not invariant under any other closed subgroup containing $M$.

\begin{theorem}\label{soporte}
Let $S$ be an FSIS 
generated by $\{\vp_1,\ldots,\vp_{\ell}\}$, and define
$$ E_j=\{\w\in\W\,: \, \dim _S(\w)=j\},\quad j=0,\ldots, \ell.$$
If $S$ is $M$-invariant and $\SR'$ is any measurable section of $\R^d/M^*$, then
$$|\{ y\in \SR'\,: \, \widehat{\vp_h}(y)\neq 0\}|\leq \sum_{j=0}^\ell j|E_j|\leq \ell,$$
for each $h=1,\ldots, \ell$.
\end{theorem}

\begin{proof}
The measurability of the sets $E_j$ follows from the results
of Helson \cite{Hel64}, e.g., see \cite{BK06} for an argument of this type.

Fix any $h\in\{0,\ldots,\ell\}$.  Note that, as a consequence of Remark \ref{fibra-cut-off},
if $J_{U_{\s}}(\w)=\{0\}$, then $\widehat{\vp_h}(\w+\s+m^*)=0$ for all $m^*\in M^*$.

On the other hand, since $\{\W+\s+m^*\}_{\s\in\SN, m^*\in M^*}$ is a partition of $\R^d$, if $\w\in\W$ and $\s\in\SN$ are fixed, there exists a unique $m^*_{(\w,\s)}\in M^*$ such
that $\w+\s+m^*_{(\w,\s)}\in\SR'$.

So,
$$\{\s\in \SN\,:\,\widehat{\vp_h}(\w+\s+m^*_{(\w,\s)})\neq 0\}\subset \{\s\in \SN\,: \, \dim _{U_{\s}}(\w)\neq 0\}.$$
Therefore
\begin{eqnarray*}\label{dim}
\#\{\s\in \SN\,:\,\widehat{\vp_h}(\w+\s+m^*_{(\w,\s)})\neq 0\}&\leq& \#\{\s\in \SN\,: \, \dim _{U_{\s}}(\w)\neq 0\}\\
&\leq&\sum_{\s\in \SN}\dim_{U_{\s}}(\w)\\
&=&\dim _S(\w).
\end{eqnarray*}

Consequently, by Fubini's Theorem,
\begin{align*}
|\{y\in \SR'\,:\,\widehat{\vp_h}(y)\neq 0\}|&= \sum_{\s\in \SN} |\{\w\in \W\,:\,\widehat{\vp_h}(\w+\s+m^*_{(\w,\s)})\neq 0\}|\\
&=|\{(\w,\s)\in \W\times \SN\,: \, \widehat{\vp_h}(\w+\s+m^*_{(\w,\s)})\neq 0\}|\\
&=\int_{\W}\#\{\s\in \SN\,:\,\widehat{\vp_h}(\w+\s+m^*_{(\w,\s)})\neq 0\}\, dw\\
&\leq \int_{\W} \dim _{S}(\w) dw=\sum_{j=0}^{\ell} j |E_j|\leq \ell.
\end {align*}

\end{proof}

When $M$ is not discrete, the previous theorem shows that, despite the fact that $\SR'$ has infinite measure, the support of $\widehat{\varphi_h}$ in $\SR'$  has finite measure.

On the other hand, if $M$ is discrete,  the measure of $\SR'$ is equal to the measure of the section $\SR$ given by (\ref{R-def}). That is
$$|\SR'|=|\SR|=a_1\ldots a_d,$$
where $a_1,\ldots, a_d$ are the invariant factors. Thus, if $a_1\ldots a_d-\ell>0$, it follows that
\begin{equation}\label{medida-D'}
|\{ y\in \SR'\,: \, \widehat{\vp_h}(y)= 0\}|\geq a_1\ldots a_d-\ell.
\end{equation}

\begin{corollary}\label{vanish}
Let $\varphi\in L^2(\R^d)$ be given. If the SIS $S(\varphi)$ is $M$-invariant for some closed subgroup $M$ of $\R^d$ such that $\Z^d \subsetneqq M$, then $\widehat{\varphi}$ must vanish on a set of infinite Lebesgue measure.
\end{corollary}

\begin{proof}

Let $\SR$ be the measurable section of  $\R^d/M^*$ defined in (\ref{R-def}). Then,
$$\R^d=\bigcup_{m^*\in M^*}\SR+m^*,$$ thus
\begin{equation*}
|\{ y\in \R^d\,: \, \widehat{\varphi}(y)=0\}|=\sum_{m^*\in M^*}|\{ y\in \SR+m^*\,: \, \widehat{\varphi}(y)= 0\}|.
\end{equation*}
If $M$ is discrete, by (\ref{medida-D'}), we have
\begin{equation}
|\{ y\in \R^d\,: \, \widehat{\varphi}(y)=0\}|\geq \sum_{m^*\in M^*} (|\SR|-1)=+\infty.
\end{equation}
The last equality is due to the fact that $M^*$ is infinite and $|\SR|>1$ (since $M\neq\Z^d$).

If $M$ is not discrete, by Theorem \ref{soporte}, $|\{ y\in \SR+m^*\,: \, \widehat{\varphi}(y)= 0\}|=+\infty,$ hence $|\{ y\in \R^d\,: \, \widehat{\varphi}(y)=0\}|=+\infty$.
\end{proof}

It is known that on the real line, the SIS generated by a function $\vp$ with compact support can only be invariant
under integer translations. That is, $t_x\vp\notin S(\vp)$ for all $x\in\R\setminus\Z$. The following proposition extends this result to $\R^d$.

\begin{proposition}
If a nonzero function $\varphi\in L^2(\R^d)$ has compact support, then $S(\varphi)$
is not $M$-invariant for any $M$ closed subgroup of $\R^d$ such that $\Z^d\subsetneqq M$.
In particular,
\begin{equation}\label{trasl}
t_x\vp\notin S(\varphi)\quad\forall\,x\in\R^d\setminus \Z^d.
\end{equation}
\end{proposition}

\begin{proof}
The first part of the proposition is a straightforward consequence of Corollary \ref{vanish}.
To show (\ref{trasl}), take $x\in\R^d\setminus \Z^d$ and suppose that $t_x\vp\in S(\varphi)$.
If $M$ is the closed subgroup generated by $x$ and $\Z^d$, then $S(\vp)$ must be $M$-invariant, which is a contradiction.

\end{proof}

As a consequence of Theorem \ref{soporte}, in case that $M=\R^d$, we obtain the following corollary.

\begin{corollary}
If $\varphi\in L^2(\R^d)$ and $S(\varphi)$ is $\R^d$-invariant, then $$|\textnormal{supp}{(\widehat{\varphi})}|\leq 1.$$
\end{corollary}

Let $M$ be a closed subgroup of $\R^d$ containing $\Z^d$. The next theorem, states that there exists an M-invariant space $S$ that is {\it not} invariant under any  vector outside $M$. We will say in this case that $S$ is {\it exactly} $M$-invariant.

Note that because of Proposition  \ref{M-isgroup},  an M-invariant space is exactly $M$-invariant if and only if it is not invariant under any closed subgroup $M'$ containing $M.$

\begin{theorem}
For each closed subgroup $M$ of  $\R^d$ containing $\Z^d$, there exist a shift-invariant space of $L^2(\R^d)$
which is exactly $M$-invariant.
\end{theorem}
\begin{proof}

Let $M$ be a subgroup of $\R^d$ containing $\Z^d$. We will construct a principal shift-invariant space that is exactly $M$-invariant.

Suppose that  $0\in\SN$ and take $\vp\in L^2(\R^d)$ satisfying $\supp(\widehat{\vp})=B_0$, where $B_0$ is defined as in (\ref{B-sigma}). Let $S=S(\varphi)$.

Then, $U_0=S$ and $U_{\s}=\{0\}$ for $\s\in\SN$, $\s\neq 0$. So,  as a consequence of Theorem \ref{teo-subs}, it follows that $S$ is $M$-invariant.

Now, if $M'$ is a closed subgroup such that $M\subsetneqq M'$, we will show that $S$ can not be $M'$-invariant.

Since $M\subset M'$, $(M')^*\subset M^*$. Consider  a section $\mathcal{H}$ of the quotient $M^*/(M')^*$ containing the origin. Then,
the set given by $$\SN':=\{\s+h\,:\,\s\in\SN,\, h\in \mathcal{H}\},$$
is a section of $\Z^d/(M')^*$ and $0\in\SN'$.

If $\{B'_{\gamma}\}_{\gamma\in\SN'}$ is the partition defined in (\ref{B-sigma}) associated to $M'$, for each $\s\in\SN$
it holds that $\{B'_{\s+h}\}_{h\in\mathcal{H}}$ is a partition of $B_{\s}$, since
\begin{equation}\label{B'parteB}
B_{\s}=\W+\s+M^*=\bigcup_{h\in\mathcal{H}}\W+\s+h+(M')^*=\bigcup_{h\in\mathcal{H}}B'_{\s+h}.
\end{equation}

We will show now that $U'_0\nsubseteq S,$ where $U'_0$ is the subspace defined in (\ref{U-sigma}) for $M'$.
Let $g\in L^2(\R^d)$ such that $\widehat{g}=\widehat{\vp}\chi_{B'_0}$. Then $g\in U'_0.$ 
Moreover, since $\supp(\widehat{\vp})=B_0$, by (\ref{B'parteB}), $\widehat{g}\neq 0$. 

Suppose that $g\in S$, then $\widehat{g}=\eta\widehat{\vp}$ where $\eta$ is a $\Z^d$-periodic function.
Since $M\subsetneqq M'$, there exists $h\in\mathcal{H}$ such that $h\neq 0$.
By (\ref{B'parteB}), $\widehat{g}$ vanishes in $B'_h$.  Then, the $\Z^d$-periodicity of $\eta$ implies that $\eta(y)=0$ a.e. $y\in\R^d$. So $\widehat{g}=0$, which is a contradiction.

This shows that $U'_0\nsubseteq S$. Therefore, $S$ is not $M'$-invariant.

\end{proof}

\thispagestyle{empty}

\end{document}